\documentclass[10pt,oneside,a4paper]{article}
\usepackage{titlesec}
	\titleformat{\subsection}[runin]{\normalfont\bfseries}{\thesubsection.}{.5em}{}[. ~]
		\titlespacing{\subsection}{0pt}{1.5ex plus .1ex minus .2ex}{0pt}
\usepackage{graphicx}
\usepackage{subcaption}
\usepackage{wrapfig}
\usepackage{amsfonts,amssymb,amsmath,amsthm}
\usepackage{enumerate}
\usepackage{url}
\usepackage{mathrsfs}
\usepackage[dvipsnames]{xcolor}
\usepackage{hyperref}
    \hypersetup{colorlinks=true,linkcolor=Magenta
,citecolor=PineGreen}

\newtheoremstyle{thmstyle}
{\topsep}
{\topsep}
{\itshape}
{0pt}
{\bfseries}
{.}
{5pt plus 1pt minus 1pt}
{#2.\hspace{3pt}#1\thmnote{~\textnormal{(#3)}}}

\newtheoremstyle{defistyle}
{\topsep}
{\topsep}
{}
{0pt}
{\bfseries}
{.}
{5pt plus 1pt minus 1pt}
{#2.\hspace{3pt}#1#3}

\theoremstyle{thmstyle}
\newtheorem{thm}[subsection]{Theorem}
\newtheorem{lemma}[subsection]{Lemma}
\newtheorem{prop}[subsection]{Proposition}
\newtheorem{cor}[subsection]{Corollary}

\theoremstyle{defistyle}
\newtheorem{defi}[subsection]{Definition}
\newtheorem{rmk}[subsection]{Remark}

\newcommand{\tance}{\mathop{\mathrm{ta}}}
\newcommand{\Arg}{\mathop{\mathrm{Arg}}}
\newcommand{\Area}{\mathop{\mathrm{Area}}}

\newcommand{\SU}{\mathop{\mathrm{SU}}}
\newcommand{\PU}{\mathop{\mathrm{PU}}}
\newcommand{\BV}{{\mathrm B}\,V}

\newcommand{\SV}{{\mathrm S}\,V}
\newcommand{\EV}{{\mathrm E}\,V}

\newcommand{\imag}{\mathop{\mathrm{Im}}}
\newcommand{\PV}{\mathbb{P}_{\mathbb C}V}
\newcommand{\dist}{\mathop{\mathrm{dist}}}

\begin{document}
\title{Basic $\PU(1,1)$-representations of the hyperelliptic group are discrete} 
\author{Felipe A. Franco}
\date{}

\maketitle

\begin{abstract}
We show that a $\PU(1,1)$-representation of 
the hyperelliptic group~$H_n$ is basic if and only if it is discrete and faithful, 
thus partially proving a conjecture  by S. Anan'in and E. 
Bento Gon\c calves~\cite{Sasha2009,ABG2007} in the case of the Poincar\'e disc.
\end{abstract}

\section{Introduction}
In~\cite{ABG2007}, S. Anan'in and E. Bento Gon\c calves obtained an elementary
description of the Teichm\"uller space of hyperelliptic surfaces (see 
Section~\ref{sec:hypgroups} for a brief discussion). It revolves 
around the study of $\PU(1,1)$-representa{\-}tions of the
{\it hyperelliptic group\/} $H_n$, the group with $n\geq 5$ generators $r_1,\ldots, r_n$
and defining relations $r_i^2=1$ and $r_n\ldots r_1=1$, where $\PU(1,1)$ is the group
of orientation preserving isometries of the Poincar\'e disc $\mathbb D$. 
Hyperelliptic surfaces and hyperelliptic groups are related as follows. 
By~\cite{Maclachlan1971}, if $\Sigma=\mathbb D/\pi_1\Sigma$ is a 
hyperelliptic surface of genus~$g\geq 2$, the extension of the fundamental group~$\pi_1\Sigma$
induced by the hyperelliptic involution of~$\Sigma$ is the group $H_n$, where 
$n=2g+2$. On the other hand,
given $n\geq 5$, if $\varrho:H_n\to\PU(1,1)$ is a discrete and faithful representation
of $H_n$, either~$n$ is even and there exists a hyperelliptic surface $\Sigma$ of genus
$g=\frac{n-2}{2}$ such that $\pi_1\Sigma$ is an index~$2$ subgroup of $H_n$, or 
$n$ is odd in which case there exists a hyperelliptic surface $\Sigma$ with genus 
$g=n-3$ and $\pi_1\Sigma$ is an index~$4$ subgroup of $H_n$ 
(see, for instance, \cite[Section~2.1.4]{AGG2011}). 

Since every involution in~$\PU(1,1)$ is a reflection in a point of $\mathbb D$, 
there is a one-to-one correspondence between the space of representations
$\varrho:H_n\to\PU(1,1)$, $\varrho(r_i)\neq 1$, and the space of
relations $R^{q_n}\ldots R^{q_1}=\pm 1$ between reflections $R^{q_i}$ in points $q_i\in\mathbb D$
(we write such relations at the level of $\SU(1,1)$). Moreover, given such a relation,
we can deform it using {\it bendings\/} (see~\cite{Sasha2012,spell}, 
also referred as 
{\it simple earthquakes\/} in~\cite{ABG2007}): for an index 
$i$ modulo~$n$, if $q_i'$ and $q_{i+1}'$ are the points in $\mathbb D$ obtained 
by moving the points $q_i$ and $q_{i+1}$ along the geodesic that joins them while 
maintaining their distance, then
$R^{q_{i+1}'}R^{q_i'}=R^{q_{i+1}}R^{q_i}$; so, changing $R^{q_{i+1}}R^{q_i}$ to 
$R^{q_{i+1}'}R^{q_i'}$ in the initial relation, we obtain a new one.
The corresponding new representation produces another hyperelliptic surface
that is a {\it bending-deformation\/} of the starting one. In this sense, 
bending-deformations  are just a particular case of Fenchel-Nielsen twist 
deformations~\cite{Hubbard2016}.

Sometimes a relation can have its length reduced. For example, if in the relation 
$R^{q_n}\ldots R^{q_1}=\pm 1$
we have $q_i=q_{i+1}$, by {\it canceling\/} $R^{q_{i+1}}R^{q_i}$ (which is the identity 
in $\PU(1,1)$), we get a new relation of smaller length $n-2$. 
Basic relations are those that cannot have their length reduced {\sl even after a finite
number of bendings}. In~\cite{Sasha2012,ABG2007} it is conjectured, both in the case 
of the Poincar\'e disc and the complex hyperbolic plane 
$\mathbb H_{\mathbb C}^2$, that there exists a finite number of basic relations and 
that the representation $\varrho:H_n\to\PU(1,1)$ (or $\PU(2,1)$) induced by a basic 
relation is discrete and faithful.

In this work, we introduce the concept of {\it concatenating\/} representations of $H_n$ 
(see Subsection~\ref{subsec:concat})
and formalize the idea of basic relations and basic representations (see 
Definition~\ref{defi:basic}). Considering the case of $\mathbb D$, we prove that a 
representation of $H_n$ is basic if and only if it is discrete and faithful 
(see Theorem~\ref{thm:main}). This proves the second half of the mentioned conjecture
and disproves the first one, since there exists a discrete and faithful representation
$\varrho:H_n\to\PU(1,1)$ for each~$n\geq 5$.

\section{Hyperbolic geometry}
\label{sec:hypgeo}

We follow the notation in~\cite{SashaGrossi2011}. Let $V$ be a $2$-dimensional 
$\mathbb C$-linear space equipped with a Hermitian
form $\langle-,-\rangle$ of signature~$+-$. We consider $\PV$
divided into {\it negative\/}, {\it isotropic\/}, and {\it positive\/} points:
$$\BV:=\{p\in\PV\mid\langle p,p\rangle <0\},\ \SV:=\{p\in\PV\mid\langle p,p\rangle=0\},
\ \EV:=\{p\in\PV\mid\langle p,p\rangle >0\}.$$
(Here and throughout this paper, we use the same notation for both elements of $\PV$ and 
its representatives in $V$.)
For nonisotropic points $p\in\PV$, we have the identification $\mathrm T_p\PV\simeq
\mathrm{Lin}_{\mathbb C}(\mathbb Cp,p^\perp)$. The Hermitian form of $V$ endows
 $\BV$ and $\EV$ with a Hermitian metric defined by
\begin{equation}
\label{eq:hermitianform}
\langle t_1,t_2\rangle :=- 4\frac{\langle t_1(p),t_2(p)\rangle}{\langle p,p\rangle},
\end{equation}
where $t_1,t_2\in\mathrm{Lin}_{\mathbb C}(\mathbb Cp,p^\perp)$. 
Tanking the real part of this metric, we equip $\BV$ and $\EV$ with their usual
Poincar\'e disc Riemannian metric. We call $\PV$ the {\it Riemann-Poincar\'e sphere\/}. 

Given two nonisotropic points $p_1,p_2\in\PV\setminus\SV$, the {\it tance\/} between 
them is defined by
$$\tance(p_1,p_2):=\frac{\langle p_1,p_2\rangle\langle p_2,p_1\rangle}
{\langle p_1,p_1\rangle\langle p_2,p_2\rangle}.$$
The distance between two points $p_1,p_2$ in $\BV$ (or in $\EV$) is a monotonic function of 
the tance, namely, $\mathop{\mathrm{cosh}}^2\big(\frac{\dist(p_1,p_2)}{2}\big)=\tance(p_1,p_2)$.
For our purposes, it is enough to work only with the disc $\BV$, 
which will be denoted by $\mathbb D$, and its absolute $\partial \mathbb D:=\SV$; we also 
denote $\overline{\mathbb D}:=\mathbb D\cup\partial\mathbb D$. 
For distinct points $p_1,p_2\in\overline{\mathbb D}$ we denote by $\mathrm G\wr p_1,p_2\wr$ 
the geodesic line through $p_1,p_2$ and by $\mathrm G[q_1,q_2]$ the closed geodesic segment 
connecting $p_1,p_2$.

\medskip

Let us quickly discuss what is the area of a region bounded
by a piecewise geodesic path. Let~$\mathbb B^2$ denote the unit closed disc in~$\mathbb C$ and
let $\varphi:\mathbb B^2\to\overline{\mathbb D}$ be a piecewise smooth function such that
$\varphi(\partial\mathbb B^2)$ is the union of finitely many geodesic segments in 
$\overline{\mathbb D}$ and such that $\varphi^{-1}(\partial\mathbb D)$ is a finite set of points in
the boundary of $\mathbb B^2$.
The imaginary part of the Hermitian metric~\eqref{eq:hermitianform} defines
a K\"ahler form on $\mathbb D$, which we denote $\omega$. Let $P$ be a K\"ahler potential
of $\omega$, i.e., $dP=\omega$. The area of the region $\varphi(\mathbb B^2)$ is given
by $\int_{\varphi(\mathbb B^2)}\omega=\int_{\partial\varphi(\mathbb B^2)}P$. 
So, for a geodesic path~$C$ with vertices $p_1,\ldots,p_n$, and given any point
$c\in\overline{\mathbb D}$, the area of the region delimited by $C$ coincides with the sum
$$\sum_{i=1}^n\Area\Delta(c,p_i,p_{i+1}),$$
where $\Delta(c,p_i,p_{i+1})$ stands for the oriented triangle with vertices
$c,p_i,p_{i+1}$ and the indices are considered modulo~$n$. This sum does not
depend on the point $c$ (see~\cite[Remark~2.3]{ABG2007}).

\medskip

The group of orientation-preserving isometries of $\mathbb D$ is $\PU(1,1)=\SU(1,1)/\{-1,1\}$,
where 
$$\mathrm{U}(1,1):=\{I\in\mathrm{GL}_{\mathbb C}V\mid \langle Iv,Iw\rangle=\langle v,w\rangle
\ \text{for every}\ v,w\in V\},$$ 
and $\SU(1,1)$ consists of the elements of $\mathrm{U}(1,1)$ with determinant~$1$.
Every isometry that we will consider is orientation-preserving.
The nonidentical isometries are classified into {\it elliptic\/}, {\it parabolic\/},
and {\it hyperbolic\/}: elliptic isometries fix one point in $\mathbb D$,
parabolic isometries fix exactly one point in $\partial\mathbb D$, and hyperbolic isometries
fix exactly two points in $\partial\mathbb D$.

Every elliptic isometry can be written as
$$R_{\alpha}^p:x\mapsto (\overline\alpha-\alpha)\frac{\langle x,p\rangle}
{\langle p,p\rangle}+\alpha x,$$
where $p\in\mathbb D$ and $\alpha\in\mathbb C$ is such that
$|\alpha|=1$ and $\alpha\neq\pm 1$. We call $p$ the {\it center\/} and $\alpha$ 
the {\it parameter\/} of the elliptic isometry $R_\alpha^p$. 
The isometry $R_\alpha^p$ acts as a rotation around $p$ by the angle $\Arg\alpha^2$.
Every (nonidentical) involution in $\PU(1,1)$ is a reflection $R_{\mathrm{i}}^{p}$ in 
a point $p\in\mathbb D$. We denote the reflection $R_{\mathrm{i}}^p$, where $p\in\mathbb D$, simply 
by~$R^p$.
(Note that $R_{-\mathrm{i}}^p=-R_{\mathrm{i}}^p$, so both $R_{\pm \mathrm{i}}^p$ are 
possible lifts to $\SU(1,1)$ of an
involution fixing $p$; we choose the lift with parameter~$\mathrm{i}$.)

We are interested in relations between reflections. In $\SU(1,1)$, 
these relations have the form $R^{q_n}\ldots R^{q_1}=\varepsilon$,
where $\varepsilon=\pm 1$. 
Note that $R^pR^p=-1$, which is the identity isometry at the level 
of $\PU(1,1)$; we call such relations {\it cancellations}\/. Every length~$2$
relation between reflections is a cancellation. There are no length~$3$ relation
between reflections.

\section{The hyperelliptic group}
\label{sec:hypgroups}

Given $n\geq 5$, we consider the {\it hyperelliptic group} 
$$H_n:=\langle r_1,\ldots,r_n\mid r_i^2=r_n\ldots r_1=1\rangle.$$
Denote by $\mathcal PH_n$ the space of representations $\varrho:H_n\to\PU(1,1)$,
with $\varrho(r_i)\neq 1$ for all indices~$i$. As seen in Section~\ref{sec:hypgeo}, 
every (orientation-preserving) involution of $\mathbb H_{\mathbb C}^1$ is a reflection in a 
point. Thus, given $\varrho\in\mathcal PH_n$, there exist $q_1,\ldots, q_n\in\mathbb D$ such that 
$\varrho(r_i)=R^{q_i}$ for every~$i$. The defining relations of $H_n$ imply
$R^{q_n}\ldots R^{q_1}=\varepsilon$, $\varepsilon=\pm1$,
and this expression is written in terms of~$\SU(1,1)$.

On the other hand, every length~$n$ relation $R^{q_n}\ldots R^{q_1}=\varepsilon$, 
$q_i\in\mathbb D$, produces a representation $\varrho\in\mathcal PH_n$, 
defined by $\varrho(r_i)=R^{q_i}$.
So, there is a one-to-one correspondence between $\mathcal PH_n$ and
the space of length~$n$ relations between reflections.

For $n=2$, we call $H_2$ the {\it cancellation group\/} since every representation
$\varrho:H_2\to\PU(1,1)$, $\varrho(r_i)\neq 1$, of it is given by a cancellation; as before,
the space of such representations is denoted~$\mathcal PH_2$ and we also call its elements
cancellations. 
In what follows, unless otherwise stated, whenever we are considering the group $H_n$, either
$n\geq 5$ and $H_n$ is a hyperelliptic group, or $n=2$ and we are dealing with the cancellation 
group~$H_2$.  

Let $\varrho\in\mathcal PH_n$, and let $q_1,\ldots,q_n\in\mathbb D$ be such that 
$\varrho(r_i)=R^{q_i}$. Given a point $c_0\in\overline{\mathbb D}$, define 
$c_i:=R^{q_i}c_{i-1}$, indices considered modulo~$n$ (clearly, $c_n=c_0$). The list of points 
$c_0,\ldots, c_{n-1}$ is the {\it cycle\/} of $c_0$ under~$\varrho$. Consider the piecewise
geodesic path $C$ given by the cycle of~$c_0$ under~$\varrho$, i.e., the closed path given 
by geodesic segments $\mathrm G[c_i,c_{i+1}]$. 
The oriented area of the polygonal region bounded by such path $C$ (as described in
Section~\ref{sec:hypgeo}) is the {\it area\/} of the representation~$\varrho$ and 
is denoted by~$\Area\varrho$. This area is well defined since it does not depends on $c_0$ 
(see~\cite[Lemma~3.2]{ABG2007}). 
Clearly, the area of a cancellation vanishes.

\begin{prop}[see~\cite{ABG2007}]
\label{prop:arearep}
Let $\varrho\in\mathcal PH_n$, where $n\geq 5$. Then:

$\bullet$ $\Area\varrho=n\pi\ (\mathrm{mod}\,2\pi)$;

$\bullet$ $-(n-4)\pi\leq\Area\varrho\leq(n-4)\pi$. 
\end{prop}

We say that a representation $\varrho\in\mathcal PH_n$, $n\geq 5$, is {\it maximal\/} if 
$\Area\varrho=\pm(n-4)\pi$. Clearly, two representations in 
$\mathcal RH_n$ that are $\PU(1,1)$-conjugated have the same area.

Let~$\varrho\in\mathcal PH_n$. Denote by $\varrho^-\in\mathcal PH_n$ the
representation given by $\varrho^-(r_i)=\varrho(r_{n+1-i})$. Then 
$\Area\varrho^-=-\Area\varrho$. In terms of relations, if $\varrho$ is given by
the relation $R^{q_n}\ldots R^{q_1}=\varepsilon$, then $\varrho^{-}$ is given by 
the relation $R^{q_1}\ldots R^{q_n}=\varepsilon$
($R^{q_n}\ldots R^{q_1}=\varepsilon$ clearly implies $R^{q_1}\ldots R^{q_n}=
\varepsilon$).

\subsection{Bendings}
Take distinct points $q_1,q_2\in\mathbb D$. If $q_1',q_2'$ are obtained
by moving $q_1,q_2$ along the geodesic $\mathrm G\wr q_1,q_2\wr$, maintaining
the distance between them, i.e., $\tance(q_1,q_2)=\tance(q_1',q_2')$, then 
$R^{q_2'}R^{q_1'}=R^{q_2}R^{q_1}$; a relation of this form is called a {\it bending relation\/} 
(or {\it simple earthquake\/}). Every length $4$ relation is a bending relation 
(see for instance~\cite{Sasha2012} and~\cite{spell}).

Given a product $R^{q_2}R^{q_1}$, there exists a one parameter subgroup $B:\mathbb R
\to\PU(1,1)$ that performs every bending, i.e., such that: 

\smallskip

$\bullet$ $B(s)$ lies in the centralizer of $R^{q_2}R^{q_1}$ for every $s$; 

\smallskip

$\bullet$ given $q_1',q_2'\in\mathbb D$ with $R^{q_2'}R^{q_1'}=R^{q_2}R^{q_1}$, there
exists $s$ satisfying $B(s)q_i=q_i'$, $i=1,2$. 

\smallskip

\noindent The orbit of $q_1$ (or $q_2$) under such group is the geodesic $\mathrm G\wr q_1,q_2\wr$.

Bendings can be seen as one parameter deformations of representations 
$\varrho\in\mathcal PH_n$. In fact, 
suppose that $\varrho$ is given by the relation $R^{q_n}\ldots R^{q_1}=\varepsilon$.
If $q_i',q_{i+1}'\in\mathbb H_\mathbb C^1$ are such that $R^{q_{i+1}'}R^{q_i'}
=R^{q_{i+1}}R^{q_i}$, we can substitute $R^{q_{i+1}}R^{q_i}$ by $R^{q_{i+1}'}R^{q_i'}$
in the given relation, obtaining a new relation 
$R^{q_n}\ldots R^{q_{i+1}'}R^{q_i'}\ldots R^{q_1}=\varepsilon$.

If $\varrho$ is given by the relation $R^{q_n}\ldots R^{q_1}=\varepsilon$, we will say that 
a representation $\widehat\varrho$ obtained after finitely many 
bendings of $R^{q_n}\ldots R^{q_1}=\varepsilon$ is a {\it bending-deformation}\/ of
$\varrho$. We will also say that $\varrho$ and $\widehat\varrho$ are {\it bending-connected}\/.

\begin{prop}[see~\cite{ABG2007}]
\label{prop:bendarea}
Bending-deformations preserve the area of a representation. Moreover, two 
maximal representations of the same hyperelliptic group are bending-connected.
\end{prop}

A cycle of pairwise distinct isotropic points $p_1,\ldots, p_m\in\partial\mathbb D$, $m\geq 3$, 
is {\it positive\/} (resp. {\it negative\/}) if points are listed 
in counterclockwise (resp. clockwise) sense in the circumference 
$\partial\mathbb D$.

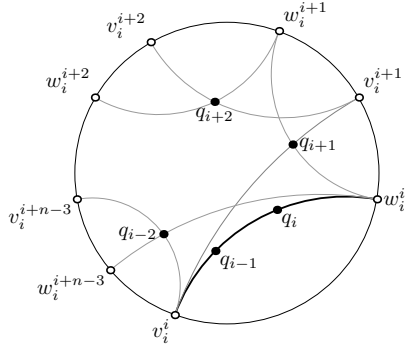
\begin{figure}[h!]
\centering
\includegraphics[scale=.8]{figs/icycle.mps}
\end{figure}

Let $\varrho\in\mathcal PH_n$, $n\geq 5$, be a representation such that $\varrho(r_ir_{i-1})$
is hyperbolic, for every index~$i$, i.e., if $q_i\in\mathbb D$ are points
such that $\varrho(r_i)=R^{q_i}$, then $q_i\neq q_{i+1}$, for every~$i$. 
Given an index~$i$, denote by $v_i,w_i\in\partial\mathbb D$ the attractor and repeller, 
respectively, of the hyperbolic isometry $R^{q_i}R^{q_{i-1}}$. In particular, 
$\mathrm G\wr q_{i-1},q_i\wr=\mathrm G\wr v_i,w_i\wr$. Define $v_i^i:=v_i,w_i^i:=w_i$ and $v_i^j:=R^{q_j}v_i^{j-1}, w_i^j:=R^{q_j}w_i^{j-1}$.
The list $v_i^i,w_i^i,v_i^{i+1},w_i^{i+1},\ldots, v_i^{i+n-3},w_i^{i+n-3}$ is the 
{\it $i$-cycle\/} of $\varrho$.

\begin{thm}[see~\cite{ABG2007}]
\label{thm:icyclediscr}
Let $\varrho\in\mathcal PH_n$, $n\geq 5$. The following are equivalent:

$\bullet$ $\varrho$ is maximal;

$\bullet$ the $i$-cycle of $\varrho$ is positive/negative;

$\bullet$ $\varrho$ is faithful and discrete.
\end{thm}

We denote by $\mathcal RH_n$ the space of discrete and faithful representations 
$\varrho\in\mathcal PH_n$, and by $\mathcal TH_n$ the Teichm\"uller space of 
$H_n$, i.e., $\mathcal TH_n=\mathcal RH_n/\PU(1,1)$ (the space of discrete
and faithful representations of $H_n$ up to $\PU(1,1)$-conjugation).
By Proposition~\ref{prop:bendarea} and Theorem~\ref{thm:icyclediscr}, the
space $\mathcal TH_n$ has two connected components $\mathcal TH_n^\pm$, corresponding
to representations of maximal area $\pm(n-4)\pi$, and bending-deformations act
transitively in each of these components. 

\smallskip

Now, we present the description of the space $\mathcal TH_n^\pm$
obtained in~\cite{ABG2007}. Consider $\overline{\mathbb D}$ identified with the closed
disc $\{z\in\mathbb C\mid |z|\leq 1\}$ and denote by $\mathbb S^1_+$ the upper
semicircle $\{z\in\partial\mathbb D\mid\imag z>0\}$.
Let~$n\geq 5$. The space $\mathcal TH_n^+$ can be parametrized as
the space of points $z_1,\ldots, z_{2n-6}\in\mathbb S^1_+$ that form a 
positive cycle in $\partial\mathbb D$. (For the description of $\mathcal TH_n^-$,
we consider negative cycles in the lower semicircle $\mathbb S^1_-$.)
In fact, on one hand we have Theorem~\ref{thm:icyclediscr}, and we can conjugate 
the representation so that the $i$-cycle $v_i^i,w_i^i,v_i^{i+1},w_i^{i+1},\ldots,
v_i^{i+n-3},w_i^{i+n-3}$ satisfies $v_i^i=-1$ and $w_i^i=1$, obtaining $2n-6$ points
in $\mathbb S^1_+$. On the other hand, given points $z_1,\ldots, z_{2n-6}\in\mathbb S^1_+$
forming a positive cycle, we can obtain a relation $R^{q_n}\ldots R^{q_1}=\pm \varepsilon$ having
as $i$-cycle the list $-1,1,z_1,\ldots, z_{2n-6}$ as follows: 
define $q_i:=0$, $q_{i+1}:=\mathrm G\wr -1,z_1\wr\cap
\mathrm G\wr 1,z_2\wr$, $q_{i+k}:=\mathrm G\wr z_{2k-3},z_{2k-1}\wr\cap
\mathrm G\wr z_{2k-2},z_{2k}\wr$ for $k=2\ldots,n-3$, and $q_{i+n-2}:=
\mathrm G\wr z_{2n-7},-1\wr\cap\mathrm G\wr z_{2n-6},1\wr$. Now we observe that the isometry
$I=R^{q_{i_n-1}}\ldots R^{q_{i+1}}R^{q_i}$ is a hyperbolic isometry with fixed 
points $\pm1\in\partial\mathbb D$ (see the proof of \cite[Corollary~3.17]{ABG2007} for details).
Hence, there exists a point $q_{i+n-1}$ in the geodesics $\mathrm G\wr -1,1\wr$ such that
$R^{q_{i+n-1}}\ldots R^{q_{i+1}}R^{q_i}=\pm1$. Note that the points $q_{i+n-1}$ must lie
to the left of $q_i$ and, since the isometry $R^{q_{i+n-1}}R^{q_i}$ depends only on
$\tance(q_{i},q_{i+n-1})$, it is computationally easy to find such point.

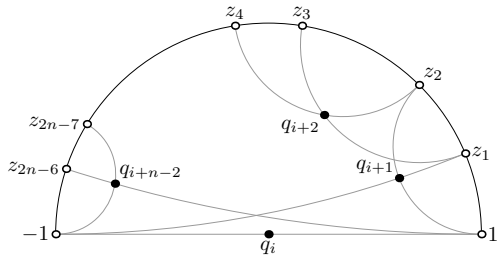
\begin{figure}[h!]
\centering
\includegraphics[scale=.8]{figs/icycle2.mps}
\end{figure}

It should be noted that other descriptions for the Teichm\"uller space of hyperelliptic surfaces
can be found in \cite{DB2005} (for the case of genus $g=2$) and in \cite{haeringen2020} for the
general case, where they associate to each hyperelliptic surface an 
{\it admissible\/} or {\it hyperelliptic polygons\/}. In this language, the process of taking
$2n-6$ in cyclic order in the semicircle is a way of constructing admissible polygons.

\section{Basic relations and discrete representations}

Starting with a relation $R^{q_n}\ldots R^{q_1}=\varepsilon$, where $q_1,\ldots,q_n\in\mathbb D$,
it could be that, after finitely many bendings, we arrive at a situation where we can use
cancellations to reduce the length of such relation, e.g., we can arrive at a relation
$R^{q_n'}\ldots R^{q_{i+1}'}R^{q_i'}\ldots R^{q_1'}=\varepsilon$ where
$q_{i+1}'=q_i'$; using a cancellation, we obtain a new relation
$R^{q_n'}\ldots R^{q_{i+2}'}R^{q_{i-1}'}\ldots R^{q_1'}=\varepsilon$ of length~$n-2$.
More generally, it could be that, after finitely many bendings, we obtain that the given
relation contains a relation of smaller length, i.e., is the {\it concatenation\/} of
two relations, and we can again reduce the length of the relation. We refer to relations that 
cannot have their length reduced in this way as {\it basic}.
In what follows, we want to formalize this process and understand when a relation
(or, equivalently, a representation in $\mathcal PH_n$) is basic, and what
being basic or not can tell us about the representation.

\subsection{Concatenating representations}
\label{subsec:concat}
Let $\varrho_1$ and $\varrho_2$ be two $\PU(1,1)$-representations, one of $H_m$ and the
other of $H_n$, given by relations $R^{p_m}\ldots R^{p_1}=\varepsilon_1$ and
$R^{q_n}\ldots R^{q_1}=\varepsilon_2$, respectively. Using $\varrho_1,\varrho_2$,
we obtain new representations by concatenating 
$\varrho_1$ and $\varrho_2$ in different ways: given an index $i=0,1,\ldots,n$,
we define a representation $\varrho_1\odot_i\varrho_2$ as being given by 
the relation 
$$
R^{q_n}\ldots R^{q_{i+1}}R^{p_m}\ldots R^{p_1}R^{q_i}\ldots R^{q_1}=
\varepsilon_1\varepsilon_2.$$
This well defines an operation $\mathcal PH_m\odot_i\mathcal PH_n\to\mathcal PH_{m+n}$,
called {\it concatenation\/}, for every index $i=0,1,\ldots,n$. Note that,
$$\varrho_1\odot_i(\varrho_2\odot_j\varrho_3)=\varrho_1\odot_{i+j}(\varrho_2
\odot_j\varrho_3).$$
In particular, $\varrho_1\odot_0(\varrho_2\odot_0\varrho_3)=
(\varrho_1\odot_0\varrho_2)\odot_0\varrho_3$.

\begin{defi}
\label{defi:basic}
A representation $\varrho\in\mathcal PH_n$, $n\geq 5$, is {\it basic\/} if 
it cannot be connected by finitely many bending-deformations with a
representation of the form~$\varrho_1\odot_i\varrho_2$,
for some representations $\varrho_1\in\mathcal PH_k$ and $\varrho_2\in\mathcal PH_\ell$,
where $k+\ell=n$, and some $i=0,1,\ldots, \ell$. 
\end{defi}

In terms of relations, a representation given by $R^{q_n}\ldots R^{q_1}=\varepsilon$ is 
basic if, up to finitely many bendings, this relation is not a concatenation
of relations of smaller length.

Note that, if $\varrho\in\mathcal PH_n$ is not basic then, up to $\PU(1,1)$-conjugation and
finitely many bending-deformations, we can assume $\varrho=\varrho_1\odot_0\varrho_2$. In fact,
we have $\varrho$ given by a relation of the form 
$R^{q_\ell}\ldots R^{q_{i+1}}R^{p_k}\ldots R^{p_1}R^{q_i}\ldots R^{q_1}=\varepsilon$.
Conjugating by $I:=R^{q_i}\ldots R^{q_1}$, we obtain the relation
$$R^{q_i}\ldots R^{q_1}R^{q_\ell}\ldots R^{q_{i+1}}R^{p_k}\ldots R^{p_1}=\varepsilon,$$
that produces a representation of the form $\varrho_1\odot_0\varrho_2$, where
$\varrho_1\in\mathcal PH_k$ and $\varrho_2\in\mathcal PH_\ell$.

\begin{prop}
\label{prop:areaproduct}
$\Area(\varrho_1\odot_i\varrho_2)=\Area\varrho_1+\Area\varrho_2$.
\end{prop}

\begin{proof}
Suppose that $\varrho_1\in\mathcal PH_m$ and $\varrho_2\in\mathcal PH_n$, and let 
$c_0\in\overline{\mathbb D}$. If $c_0,\ldots, c_{m+n-1}$ is the cycle of
$c_0$ under $\varrho_1\odot_i\varrho_2$, then $c_i,\ldots c_{m-1}$ is the cycle
of $c_i$ under $\varrho_1$ (in particular, $c_{i+m}=c_i$). Moreover,
$c_0,\ldots, c_i, c_{i+m+1},\ldots, c_{m+n-1}$ is the cycle of $c_0$ under $\varrho_2$.
Therefore, $\Area(\varrho_1\odot_i\varrho_2)$ is the sum of the oriented area
of the geodesic polygon with vertices $c_i,\ldots c_{m-1}$ and the one with vertices
$c_0,\ldots, c_i, c_{i+m+1},\ldots, c_{m+n-1}$, which is exactly, by definition, 
$\Area\varrho_1+\Area\varrho_2$. 
\end{proof}

\begin{cor}
\label{cor:notdf}
If\/ $\varrho\in\mathcal RH_n$, then $\varrho$ is basic.
\end{cor}

\begin{proof}
Suppose that $\varrho$ is not basic. By Propositions~\ref{prop:bendarea} 
and~\ref{prop:areaproduct}, $\Area\varrho=\Area\varrho_1+\Area\varrho_2$ for some
representations $\varrho_1\in\mathcal PH_k$ and $\varrho_2\in\mathcal PH_\ell$, $k+\ell=n$.
Thus, by Proposition~\ref{prop:arearep},
$$-(n-4)\pi<-(k+\ell-8)\pi\leq\Area\varrho\leq(k+\ell-8)\pi<(n-4)\pi,$$
and by Theorem~\ref{thm:icyclediscr}, $\varrho$ is not faithful and discrete.
\end{proof}

It is possible for the area of a representation to be not maximal while the representation
to be discrete.
For instance, consider a representation $\varrho\in\mathcal PH_5$ given by a pentagon 
$R^{p_5}\dots R^{p_1}=\pm1$. Then,
the representation $\varrho^-\odot_{i}\varrho\in\mathcal PH_{10}$ satisfies 
$\Area(\varrho^-\odot_i\varrho)=0$, for any index~$i$
(Proposition~\ref{prop:areaproduct}).
Since every pentagon is discrete, $\varrho$ is discrete which implies that
$\varrho^-\odot_i\varrho$ is discrete (but it is clearly
not faithful).

\begin{defi}
Given $\varrho\in\mathcal PH_n$ and $q_i\in\mathbb D$ such that $\varrho(r_i)
=R^{q_i}$, we say that $\mathrm G\wr q_{k-1},q_k\wr$ and $\mathrm G\wr q_{\ell-1},q_\ell\wr$ 
are {\it neighboring\/} geodesics with respect to $\varrho$ if either $i=j-1$ or $i=j$ or $i=j+1$ 
(indices are module~$n$).  
\end{defi}

\begin{lemma}
\label{lemma:nonneigh}
If non-neighboring geodesics with respect to $\varrho\in\mathcal PH_n$ intersect, 
then\/~$\varrho$ is not basic. More precisely, if\/ $q_i\in\mathbb D$
are such that $\varrho(r_i)=R^{q_i}$ then, up to finitely many bendings, we can reduce
the length of the relation $R^{q_n}\ldots R^{q_1}=\pm1$ by a cancellation.
\end{lemma}

\begin{proof}
If $\mathrm G\wr q_{i-1},q_i\wr$ intersects $\mathrm G\wr q_{i+1},q_{i+2}\wr$ then, by bending
$R^{q_{i+2}}R^{q_{i+1}}$, we arrive at a configuration where 
$q_{i+1}\in\mathrm G\wr q_{i-1},q_{i}\wr$ (here and in what follows, abusing notation, we 
denote the new centers obtained by bending $R^{q_{j+1}}R^{q_{j}}$ again by $q_{j}$,$q_{j+1}$).
Now, bending $R^{q_i}R^{q_{i-1}}$, we make $q_{i}=q_{i+1}$.

Suppose that $\mathrm G\wr q_{i-1},q_i\wr,\mathrm G\wr q_{i+k},q_{i+k+1}\wr$ are 
non-neighboring intersecting geodesics, $k\geq 2$.
By bending $R^{q_{i+k+1}}R^{q_{i+k}}$ if necessary, we can assume, without loss of generality, 
that $q_{i+k-1}$ and $q_{i+k}$ lie in opposite sides of the geodesic 
$\mathrm G\wr q_{i-1},q_i\wr$. This implies that $\mathrm G\wr q_{i-1},q_i\wr$ and 
$\mathrm G\wr q_{i+k-1},q_{i+k}\wr$ are non-neighboring intersection geodesics. Therefore,
repeating the argument if necessary, we arrive in the case of the previous paragraph.
\end{proof}

Let $p_1,p_2\in\overline{\mathbb D}$ be distinct points.
Considering the geodesic $\mathrm G\wr p_1,p_2\wr$ as oriented from $p_1$ to $p_2$, the dis 
$\mathbb D$ is divided into two half-spaces: the set $H^+(p_1,p_2)$ 
of the point that are on the side of the normal vector of 
$\mathrm G\wr p_1,p_2\wr$, and the set $H^-(p_1,p_2)$ of points on the side opposed 
to the normal vector of $\mathrm G\wr p_1,p_2\wr$. We denote 
by~$\overline H^{\pm}(p_1,p_2)$ the closure of these half-spaces 
in~$\overline{\mathbb D}$.

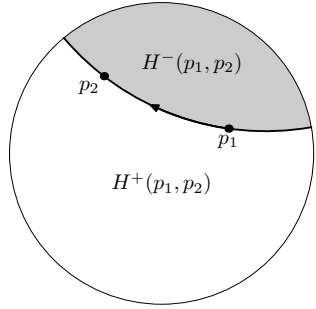
\begin{figure}[h!]
\centering
\includegraphics[scale=.8]{figs/halfspace.mps}
\end{figure}

\begin{lemma}
\label{lemma:opphalf}
Let\/ $\varrho\in\mathcal PH_n$ and\/ $q_i\in\mathbb D$ be such that\/
$\varrho(r_i)=R^{q_i}$, $i=1,\ldots, n$. If there exist indices\/ $j,k,\ell$ 
{\rm(}indices modulo~$n${\rm)} such
that\/ $q_k$ and\/ $q_{\ell}$ lie in opposite sides of\/ 
$\mathrm G\wr q_j,q_{j+1}\wr$, i.e., $q_k\in H^+(q_j,q_{j+1})$ and\/ 
$q_\ell\in H^{-}(q_j,q_{j+1})$ then, up to finitely many bendings, we can reduce the length of 
the relation\/ $R^{q_n}\ldots R^{q_1}=\pm1$ by a cancellation and\/ $\varrho$ is not basic.
\end{lemma}

\begin{proof}
Consider the point $q_{j-1}$ (index modulo~$n$). If $q_{j-1}\in
\mathrm G\wr q_j,q_{j+1}\wr$ then, by bending $R^{q_{j+1}}R^{q_j}$, we can make
$q_j=q_{j-1}$, and the result follows.

By hypothesis, if $q_{j-1}\in
H^{+}(q_{j},q_{j+1})$, then $\mathrm G\wr q_{j-1}, q_\ell\wr$ intersects
$\mathrm G\wr q_j,q_{j+1}\wr$ at a point $x$. Thus, by bending $R^{q_{j+1}}R^{q_j}$,
we can make $q_j=x$ and arrive at a configuration where 
$\mathrm G\wr q_{j-1},q_j\wr$ contains $q_{\ell}$ and, hence, intersects 
$\mathrm G\wr q_{\ell-1},q_\ell\wr$. 
Note that, since $q_\ell\notin\mathrm G\wr q_j,q_{j+1}\wr$, if 
$\mathrm G\wr q_{j-1},q_j\wr$ and $\mathrm G\wr q_{\ell-1},q_\ell\wr$ are 
neighboring geodesics with respect to~$\varrho$, then $q_{\ell-1}=q_{j-1}$; in this
case, $\mathrm G\wr q_{j},q_{j+1}\wr$ and $\mathrm G\wr q_{\ell-1},q_\ell\wr$ are
non-neighboring intersecting geodesics, and the result follows from 
Lemma~\ref{lemma:nonneigh}. But, if $\mathrm G\wr q_{j-1},q_j\wr$
and $\mathrm G\wr q_{\ell-1},q_\ell\wr$ are non-neighboring with respect
to $\varrho$, the result follows again from Lemma~\ref{lemma:nonneigh}.

Finally, if $q_{j-1}\in H^-(q_j,q_{j+1})$, we just consider $q_k$ instead of 
$q_{\ell}$, and the proof follows as in the previous case. 
\end{proof}

\begin{thm}
\label{thm:main}
A representation $\varrho\in\mathcal PH_n$ is basic iff $\varrho$ is discrete and
faithful. 
\end{thm}

\begin{proof} 
By Corollary~\ref{cor:notdf}, if $\varrho$ is discrete and faithful, then it is basic.

Conversely, suppose that $\varrho$ is not discrete or not faithful. 
Given an index~$i$, by Theorem~\ref{thm:icyclediscr}, the $i$-cycle of $\varrho$ 
is neither positive nor negative. 

First, suppose that the cycle $v_i^i,w_i^i,v_i^{i+1},w_i^{i+1},v_i^{i+2},w_i^{i+2}$
is neither positive nor negative. Since being basic (or not) is preserved by $\PU(1,1)$-conjugation, we
assume, without loss of generality, that $v_i^i=1$, $w_i^i=-1$, and $q_i=0$.
We will also assume that $q_{i+1}\notin\mathrm G\wr q_{i-1},q_i\wr$ since, 
otherwise, $\varrho$ is not basic. (In our drawing we are assuming that $q_{i+1}$ lie
in the upper half-space given by $\mathrm G\wr q_{i-1},q_i\wr$, and that the cycle 
$v_i^i,w_i^i,v_i^{i+1},w_i^{i+1}$ is positive, but the argument is analogous
in the other case.)

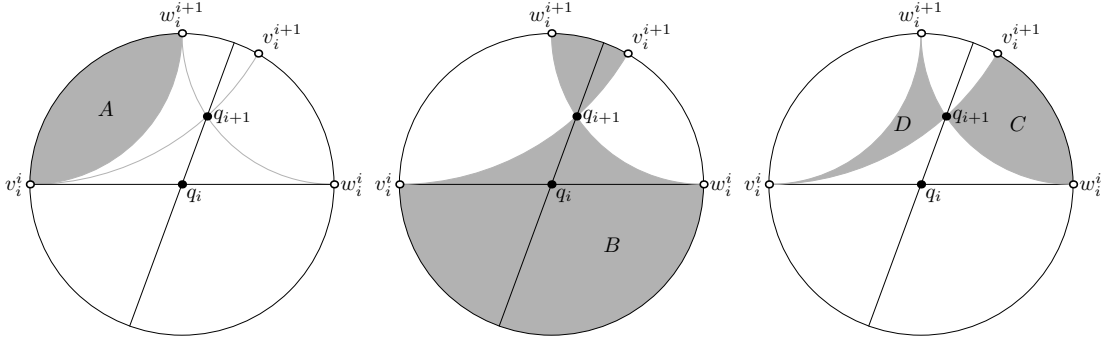
\begin{figure}[h!]
\centering
\includegraphics[scale=.8]{figs/proof-a.mps}
\includegraphics[scale=.8]{figs/proof-b.mps}
\includegraphics[scale=.8]{figs/proof-cd.mps}
\end{figure}

({\it a\/}) Note that $q_{i+2}$ does not lie in $A:=H^{+}(v_i^i,w_i^{i+1})$
(the half-space delimited by $\mathrm G\wr v_i^i,w_i^{i+1}\wr$ and not containing~$q_i$), 
otherwise the cycle $v_i^i,w_i^i,v_i^{i+1},w_i^{i+1},v_i^{i+2},w_i^{i+2}$ would be
positive;

({\it b\/}) If $q_{i+2}$ lies in the region $B$ between the geodesics 
$\mathrm G\wr v_i^i,v_i^{i+1}\wr,\mathrm G\wr w_i,w_i^{i+1}\wr$ 
containing $q_i$, i.e., if $q_{i+2}\in\big(H^{-}(v_i^i,v_i^{i+1})\cap 
H^{+}(w_i^i,w_i^{i+1})\big)\bigcup\big( H^{+}(v_i^i,v_i^{i+1})\cap H^-(w_i^i,w_i^{i+1}) \big)$,
then $\mathrm G\wr q_{i-1},q_i\wr$ and 
$\mathrm G\wr q_{i+1},q_{i+2}\wr$ are non-neighboring intersecting geodesics. 

Hence, we can assume that either:

({\it c\/}) $q_{i+2}$ lies in 
$C:=H^-(v_i^i,v_i^{i+1})\cup H^-(w_i,w_i^{i+1})$; or

({\it d\/}) $q_{i+2}$ lies in $D$, where $D$ is the closed region of finite area 
that is bounded by the triangle $\Delta(v_i^i,q_{i+1},w_i^{i+1})$;

Since we are assuming that the cycle $v_i^i,w_i^i,v_i^{i+1},w_i^{i+1}$ is 
positive, $q_{i-1}$ lies in $H^{+}(q_i,q_{i+1})$. Hence, if $q_{i+2}\in C$, then either $q_{i+2}=q_{i+1}$,
in which case $\varrho$ is not basic, or $q_{i-1}$ and $q_{i+2}$
lie in opposite sides of $\mathrm G\wr q_i,q_{i+1}\wr$. In the latter case, by
Lemma~\ref{lemma:opphalf}, $\varrho$ is not basic, which proves the theorem
in the case ({\it c\/}).

Finally, assume that $q_{i+2}\in D$. Then, $w_{i+2}^{i+2}\in
\overline H^{\,+}(v_i^i,w_i^{i+1})$ (remember that $w_{i+2}^{i+2}$ is the vertex
at infinity of $\mathrm G\wr q_{i+1},q_{i+2}\wr$ that is the attractor of 
the isometry $R^{q_{i+2}}R^{q_{i+1}}$). 
If~$q_{i+3}\in H^-(v_i^i,w_{i+2}^{i+2})$, then $\varrho$ is not basic. In fact,
by Lemma~\ref{lemma:opphalf}, we can assume that $q_{i+3}$ lies in the interior
of the geodesic quadrilateral (of finite area) with vertices 
$v_i^i,q_i,q_{i+1},w_{i+2}^{i+2}$.
In this case, by bending $R^{q_{i+2}}R^{q_{i+1}}$, we arrive at a configuration
where $\mathrm G\wr q_{i+2},q_{i+3}\wr$ intersects $\mathrm G\wr q_{i-1},q_i\wr$;
these geodesics are non-neighboring with respect to $\varrho$, otherwise $i+3=i-1$ and
$n=4$ (see Definition~\ref{defi:basic}). By Lemma~\ref{lemma:nonneigh}, $\varrho$ 
is not basic.

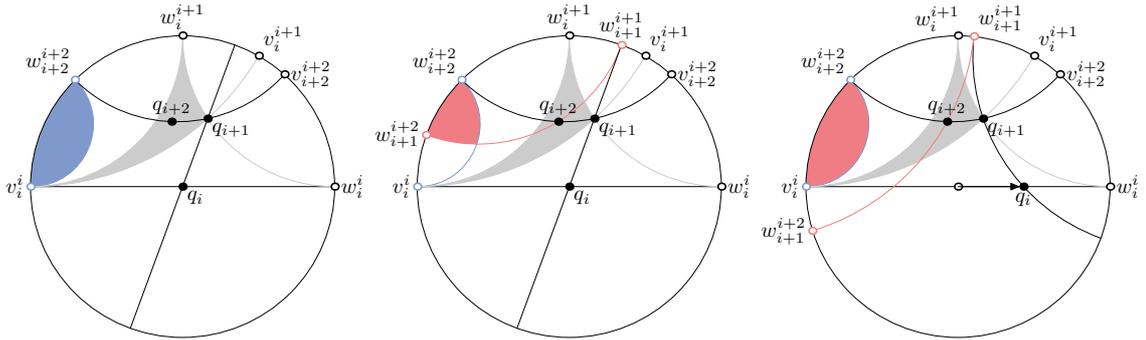
\begin{figure}[h!]
\centering
\includegraphics[scale=.8]{figs/proof-blue.mps}
\includegraphics[scale=.8]{figs/proof-red.mps}
\includegraphics[scale=.8]{figs/proof-bend.mps}
\end{figure}

Hence, we can assume that $q_{i+3}\in\overline{H}^{\,+}(v_i^i, w_{i+2}^{i+2})$. 
Note that, if $q_{i+3}\in H^{-}(w_{i+1}^{i+1},w_{i+1}^{i+2})$, then the 
non-neighboring geodesics $\mathrm G\wr q_{i+2},q_{i+3}\wr$ and 
$\mathrm G\wr q_i,q_{i+1}\wr$ intersect and $\varrho$ is not basic. Moreover,
by bending $R^{q_i}R^{q_{i-1}}$, moving $q_i$ in the direction of $w_i^i$,
we can make $w_{i+1}^{i+1}$ as close as we want to $w_{i}^{i+1}$ (considering,
momentarily, $\partial \mathbb H_{\mathbb C}^1$ equipped with the usual metric
of $\mathbb S^1$) and, since $q_{i+2}\in \overline{H}^{\,-}(v_i^i,w_{i}^{i+1})$, we 
can make
$w_{i+1}^{i+2}\in \overline{H}^{\,-}(v_i^i,w_{i}^{i+1})$, which implies that, after 
such bending of $R^{q_i}R^{q_{i-1}}$,
$\overline{H}^{\,+}(v_i^i,w_{i+2}^{i+2})\subset\overline{H}^{\,-}(v_i^i,w_{i}^{i+1})$ and
$q_{i+3}\in H^-(v_i^i,w_{i}^{i+1})$. Therefore, $\varrho$ is not basic.

\smallskip

Now, we suppose that the cycle 
$v_i^i,w_i^i,v_i^{i+1},w_i^{i+1},\ldots,v_i^{i+k},w_i^{i+k}$ is positive, but the cycle 
\begin{equation}
\label{eq:cycle}
v_i^i,w_i^i,v_i^{i+1},w_i^{i+1},\ldots,v_i^{i+k},w_i^{i+k},v_i^{i+k+1},w_i^{i+k+1}
\end{equation}
is not positive, for some $2\leq k\leq n-4$. Again, we assume $v_i^i=1$, $w_i^i=-1$,
and $q_i=0$. By hypothesis, $\mathrm G\wr q_{i-1},q_i\wr$ and 
$\mathrm G\wr q_{i+k-1},q_{i+k}\wr$ are ultraparallel. If
$q_{i+k+1}$ does not lie in $H^+(q_{i-1},q_i)\cap H^+(q_{i+k-1},q_{i+k})$, by Lemma~\ref{lemma:opphalf},
$\varrho$ is not basic; so, suppose otherwise. Note that, as in the case ({\it a}\/) above,
if $q_{i+k+1}$ lies in $H^+(v_i^i,w_{i+k}^{i+k})$, then the cycle~\eqref{eq:cycle} is positive;
moreover, by bending $R^{q_{i+k-1}}R^{q_{i+k-2}}$ if necessary, we can assume 
that $q_{i+k+1}$ does 
not lie in $\overline{H}^{\,+}(v_i^i,w_{i+k}^{i+k})$.

In this way, suppose that $q_{i+k+1}$ lies in the interior of the region delimited by the geodesic
quadrilateral $v_i^i,q_i,q_{i+k},w_{i+k}^{i+k}$. In this case, arguing as before, by bending 
$R^{q_{i+k}}R^{q_{i+k-1}}$, we arrive at a configuration where $\mathrm G\wr q_{i+k},q_{i+k+1}\wr$
intersects $\mathrm G\wr q_{i-1},q_i\wr$, and the result follows from Lemma~\ref{lemma:nonneigh}.
\end{proof}

\begin{lemma}
\label{lemma:parity}
Le\/t $\varrho\in\mathcal PH_n$, $n\geq 5$. If\/ $\Area\varrho=\pm(m-4)\pi$ for some natural number\/
$m\leq n$, then\/ $n-m$ is an even number.
\end{lemma}

\begin{proof}
By Proposition\/~\ref{prop:arearep}, $\Area\varrho=n\pi\ (\mathrm{mod}\, 2\pi)$. Thus,\/ 
$n$ and\/ $m-4$ have the same parity, which implies that\/ $n$ and\/ $m$ have the same parity. 
Therefore,\/ $n-m$ is even.
\end{proof}

\begin{cor}
\label{cor:decomp}
Let\/ $\varrho\in\mathcal PH_n$, $n\geq 5$, be a representation with\/ $\Area\varrho=\pm(m-4)\pi$,
for some natural number\/ $0\leq m\leq n$. If\/ $m\geq 5$, 
there exist\/ $\ell:=\frac{n-m}2$ {\rm(}see {\rm Lemma~\ref{lemma:parity})} 
cancellations\/~$\delta_i\in\mathcal PH_2$ and\/ $\varrho_0\in\mathcal RH_m$, 
such that
$$\varrho=\delta_1\odot_0\ldots\odot_0\delta_{\ell}\odot_0\varrho_{0},$$
up to finitely many bending-deformations and\/ $\PU(1,1)$-conjugation. 
If\/ $m=4$, there exist\/ $\ell+2$ cancellations\/ $\delta_i$ with
$$\varrho=\delta_1\odot_0\ldots\odot_0\delta_{\ell+2},$$
up to finitely many bending-deformations and\/ $\PU(1,1)$-conjugation. 
\end{cor}

The decomposition of representations in the concatenation of basic relations is not unique.
In fact, consider for example $\varrho\in\mathcal PH_5$ (this implies that $\varrho\in\mathcal
RH_5$, i.e., $\varrho$ is maximal). Thus, $\Area\varrho\odot_0\varrho^{-}=0$. By
Corollary~\ref{cor:decomp}, up to finitely many bending-deformations and
$\PU(1,1)$-conjugacy, $\varrho\odot_0\varrho^-=\delta_1\odot_0\ldots\odot_0\delta_5$, for some 
cancellations $\delta_i$.

\begin{rmk}
For any number $n\geq 5$, there exists a discrete and faithful (basic)
representation $\varrho\in\mathcal RH_n$. In fact, consider a regular right-angled\/ 
$n$-sided polyhedron in\/ $\mathbb H_{\mathbb C}^1$ with vertices\/ $p_1,\ldots, p_n$.
Then\/ $R^{p_n}\ldots R^{p_1}=1$. Moreover, the representation 
$\varrho\in\mathcal PH_n$, defined $\varrho(r_i)=R^{p_i}$, is discrete and faithful.
Note that any representation $\widehat\varrho\in\mathcal RH_n$ is, up to conjugacy, 
a bending-deformation of $\varrho$.
\end{rmk}

\begin{cor}[Corollary~3.16, \cite{ABG2007}]
$\mathcal PH_5=\mathcal RH_5$.
\end{cor}

\begin{proof}
Suppose that $\varrho\in\mathcal PH_5$, given by $R^{q_5}\ldots R^{q_1}=\varepsilon$,
is not discrete or not faithful. By the proof of Theorem~\ref{thm:main}, after
finitely many bendings, such relation can be reduced to a length~$3$ relation by 
a cancellation. But there are no length~$3$ relations between involutions
in~$\PU(1,1)$, a contradiction.
\end{proof}

\bibliographystyle{plain}
\bibliography{references-hyper}

\noindent
{\sc Felipe A.~Franco}

\noindent
{\sc Departamento de Matem\'atica, ICMC, Universidade de S\~ao Paulo, Brasil}

\noindent
\url{felipefranco@usp.br}, \url{f.franco.math@gmail.com}
\end{document}